\documentclass[a4paper,10pt]{article}

\usepackage[utf8x]{inputenc}
\usepackage[russian]{babel}

\usepackage[active]{srcltx}

\usepackage{amsmath}
\usepackage{amssymb}
\usepackage{amsthm}
\usepackage{euscript}
\usepackage{bussproofs}

\title{Об усилении теоремы Котлярского}
\author{Е.\,В. Дашков}

\begin{document}
\theoremstyle{plain}
\newtheorem{lemma}{Лемма}[section]
\newtheorem{theorem}[lemma]{Теорема}
\newtheorem{cor}[lemma]{Следствие}
\newtheorem*{stat}{Утверждение}
\theoremstyle{definition}
\newtheorem{rem}[lemma]{Замечание}
\theoremstyle{plain}

\renewcommand{\bfdefault}{b}
\renewcommand{\phi}{\varphi}
\newcommand{\T}{\mathsf{T}}
\newcommand{\lng}{\EuScript{L}}
\newcommand{\St}{\mathrm{St}_{\lng}}
\newcommand{\At}[1]{At_{\EuScript{L}}({#1})}
\newcommand{\dm}[1]{\langle #1 \rangle}
\newcommand{\gn}[1]{\left\ulcorner {#1} \right\urcorner}
\newcommand{\sys}[1]{\ensuremath{\mathbf{#1}}}
\newcommand{\glpl}{\sys{GLP_+^{=}(\mathrm \Lambda)}}
\newcommand{\con}[1]{\mathrm{Con}({#1})}
\newcommand{\ocn}[2]{\omega\text{-}\mathrm{Con}^{#2}({#1})}
\newcommand{\RFN}[3]{\mathrm{RFN}^{#3}_{#1}({#2})}
\newcommand{\imp}{\mathrel{{\dot\to}}}
\newcommand{\eq}{\mathrel{{\dot\leftrightarrow}}}
\newcommand{\all}{\dot\forall}
\newcommand{\ex}{\dot\exists}
\newcommand{\Ng}{\dot\neg}
\newcommand{\Dj}{\dot\vee}
\newcommand{\cj}{\dot\wedge}
\newcommand{\ovr}[1]{\overline{{#1}\strut}}

\maketitle

\begin{abstract}
Kotlarski's theorem\footnote{H. Kotlarski. \textit {Bounded Induction and Satisfaction Classes.} Mathematical Logic Quarterly, vol. 32, {\bf 31-34},
1986, P. 531--544.} formalized in $\sys{WKL}_0$. The English version will follow soon.
\end{abstract}

Пусть $S$ есть некоторое элементарно аксиоматизируемое расширение теории $\sys{PRA}$ в языке $\lng$. Для $n \in \omega$ определим формулы $\Gamma^S_n$, выражающие выводимость в теории $S$ с $n$-кратным применением $\omega$-правила.
$$\begin{array}{rcl}
\Gamma^S_0[\alpha] &\eqcirc& \Box_S[\alpha];\\
\Gamma^S_{n+1}[\alpha] &\eqcirc& \exists \psi\, (\forall y\, \Gamma^S_n[\psi] \wedge \Gamma^S_0[\forall y\, \psi \to \alpha]).
\end{array}
$$

\begin{lemma}\label{gamma}\rule{1pt}{0mm} Для всех $n \in \omega$ верно
\begin{enumerate}
\item Если $S \vdash \phi$, то $\sys{PRA} \vdash \Gamma^S_0[\phi]$;
\item $\sys{PRA} \vdash \forall \alpha\, (\Gamma^S_n[\phi] \to \Gamma^S_{n+1}[\alpha])$;
\item $\sys{PRA} \vdash \forall \alpha \forall \beta\, (\Gamma^S_n[\alpha \to \beta] \to (\Gamma^S_n[\alpha] \to \Gamma^S_n[\beta]))$.
\end{enumerate}
\end{lemma}
\begin{proof}
Индукция по $n$.
\end{proof}

Добавим в язык $\lng$ новый одноместный предикатный символ $\T$ и получим язык $\lng(\T)$. В этом языки определим следующие \emph{условия Тарского} $\mathrm{Tarski}$ для предиката $\T$.
\begin{enumerate}
\item $\forall x\, (\T(x) \to \St(x))$;
\item $\forall \alpha \in \Delta_0\, (Tr_0[\alpha] \leftrightarrow \T[\alpha])$;
\item $\forall \alpha\, (\T[\neg\alpha] \leftrightarrow \neg\T[\alpha])$;
\item $\forall \alpha, \beta\, (\T[\alpha \wedge \beta] \leftrightarrow \T[\alpha] \wedge \T[\beta])$;
\item $\forall \alpha\, (\T[\forall v_i\, \alpha] \leftrightarrow \forall v_i\, \T[\alpha])$ для всех $i \in \omega$,
\end{enumerate}

Теория $U$ является расширением теории $\sys{PRA} + \mathrm{Tarski}$ в языке $\lng(\T)$ произвольным множеством предложений языка $\lng$. \emph{Глобальным принципом рефлексии} для теории $S$ назовем $\lng(\T)$-предложение 
$$R_{\omega+1}(S) \eqcirc \forall\alpha (\Box_S[\alpha] \to \T[\alpha]).$$

\begin{theorem}[Kotlarski]\label{ktl}
Пусть $U$ есть расширение \sys{FT} арифметическими аксиомами. Тогда $U + R_{\omega+1}(S) \equiv_\omega U + \{\neg\Box^S_n[\bot] \mid n < \omega\}$, причем это утверждение доказуемо в $\sys{WKL}_0$.
\end{theorem}

Пусть имеется примитивно-рекурсивная геделева нумерация языка $\lng(\T, C)$, где $C$ означает счетное множество новых констант: $C = \{c_n \mid n < \omega\}$. Ясно, что система $\sys{RCA}_0$ доказывает существование соответствующих примитивно-рекурсивных "<синтаксических"> предикатов, в частности $\Box_{\sys{PC}}$ для исчисления предикатов $\sys{PC}$.

Обозначим $P \rightleftharpoons U + R_{\omega+1}(S)$ и $Q \rightleftharpoons U + \{\neg\Box^S_n[\bot] \mid n < \omega\}$ и будем рассматривать $P$ и $Q$ как (счетные) множества аксиом.

\begin{lemma} В $\sys{WKL}_0$ доказуемо
	$\sys{FT} + R_{\omega+1}(S) \vdash \forall \phi\,(\Box^S_n[\phi] \to \T[\phi]))$ для всех $n < \omega$.
\end{lemma}
\begin{proof}
Если $P \vdash \bot$, то утверждение тривиально. Пусть не так. Индукцией по $n$ покажем, что
$$U + R_{\omega+1}(S) \vdash \forall \phi\,(\Box^S_n[\phi] \to \T\phi))\mbox{ для всех $n < \omega$.}$$
При $n=0$ утверждение очевидно. Предположим $\forall \phi\,(\Box^S_k[\phi] \to \T[\phi]))$ при всех $k \leqslant n$. Пусть $\Box^S_{n+1}[\phi]$. Тогда по определению получаем
$$\exists \psi\, (\forall u\, \Box^S_n[\psi(u)] \wedge \Box^S_0[\forall x\, \psi(x) \to \phi]).$$
Применяя предположение индукции для $k = 0$ и $n$, выводим
$$\exists \psi\, (\forall u\, \T [\psi(u)] \wedge \T [\forall x\, \psi(x) \to \phi]).$$
Из условий Тарского следует $\T [\phi]$.

Предположим $\Box^S_n[\bot]$ для некоторого $n$. По доказанному, в $P$ отсюда выводится $\T[\bot]$ и, затем, $Tr_0[\bot]$. Значит, $P$ противоречива, что неверно.

Это рассуждени легко формализуемо в $\sys{RCA}_0$.
\end{proof}

Отсюда следует, что $\sys{WKL}_0 \vdash P \supseteq Q$. Предположим, что теория $Q$ несовместна. Тогда несовместна и $P$, что доказывает теорему. Будем считать, что $Q$ совместна.

Для произвольного непустого $X \subseteq \omega$ обозначим $C\upharpoonright X = \{c_n \mid n \in X\}$. Мы будем рассматривать некоторый первопорядковый язык $\lng$ (в частности, $\lng(T)$) и его естественные расширения константами из $C$. Если были добавлены константы со множеством номеров $X$, то обозначим такое расширение $\lng(C\upharpoonright X)$.

\emph{Счетной моделью} языка $\lng$ называется пара $(M, \mathcal M)$, где \emph{носитель модели} $M = |\mathcal M|$ таков, что $\varnothing \neq M \subseteq \omega$, и отображение $\mathcal M\colon St_{\lng(C \upharpoonright M)} \cup Tm_{\lng(C \upharpoonright M)} \to M \cup \{0, 1\}$ удовлетворяет следующим естественным свойствам:
\begin{itemize}
\item если $\sigma \in St_{\lng(C \upharpoonright M)}$, то $\mathcal M(\sigma) \in \{0, 1\}$, и если $t \in Tm_{\lng(C \upharpoonright M)}$, то $\mathcal M(\sigma) \in M$, причем $\mathcal M (c_u) = u$ для всех $u \in M$;
\item если $t_1, t'_1 \ldots t_n, t'_n \in Tm_{\lng(C \upharpoonright M)}$ и $\mathcal M (t_i) = \mathcal M (t'_i)$ для всех $i$, то\linebreak
$\mathcal M (\alpha(t_1,\ldots,t_n)) = \mathcal M (\alpha(t'_1,\ldots,t'_n))$, для всякого предикатного или функционального знака $\alpha$ языка $\lng$;
\item $\mathcal M (\neg\sigma) = 1 - \mathcal M (\sigma)$, $\mathcal M (\sigma \wedge \tau) = \min(\mathcal M (\sigma), \mathcal M (\tau))$ и $\mathcal M (\exists x\, \sigma(x)) = \max_{a \in M} \mathcal M (\sigma(c_a))$.
\end{itemize}
Если для $\sigma \in St_\lng$ оказывается $\mathcal M (\sigma) = 1$, то пишем $\mathcal M \models \sigma$.

\emph{Типом с $k$-параметрами} в счетной модели $\mathcal M$ языка $\lng$ называется последовательность $(\phi_n(x, b_1,\ldots, b_k))_{n \in \omega}$ формул языка $\lng(C \upharpoonright M)$, не содержащих иных свободных переменных, кроме $x$ и таких, что параметры $b_i \in C \upharpoonright M$ и $\phi(x, y_1, \ldots, y_k) \in \lng$. Тип \emph{рекурсивен}, если существует рекурсивная функция $\{g\}\colon (\omega)^{k+1} \to \omega$ такая, что $\{g\}(n, u_1,\ldots, u_k) = \gn{\phi_n(x, c_{u_1},\ldots, c_{u_k})}$ для любых $n, u_1,\ldots, u_k \in \omega$. Скажем, что $\{g\}$ нумерует тип с параметрами $(\phi_n(x, b_1,\ldots, b_k))_{n \in \omega}$.

Модель $\mathcal M$ \emph{рекурсивно насыщена}, если найдется (тотальная) рекурсивная функция $\{f\}\colon (\omega)^{k+1} \to \omega$, такая что для всякого рекурсивного типа $(\phi_n(x, b_1,\ldots, b_k))_{n \in \omega}$, где $b_i = c_{u_i}$, из
$$\forall m\, \exists a \in M\, \forall n \leqslant m\, \mathcal M(\phi_n(a, b_1,\ldots, b_k)) = 1$$
следует $\{f\}(g, u_1,\ldots, u_k) \in M$ и
$$\forall n\, \mathcal M(\phi_n(c_{\{f\}(g, u_1,\ldots, u_k)}, b_1,\ldots, b_k)) = 1.$$

\begin{lemma}[Simpson\footnote{Stephen G. Simpson. \textit{Subsystems of Second Order Arithmetic.} 2009}, p. 380]\label{rec_sat}
В $\sys{WKL}_0$ доказуемо существование счетной рекурсивно насыщенной модели произвольного счетного совместного множества $X$ формул языка $\lng$.
\end{lemma}

\begin{cor}\label{rec_sat_eq}
В $\sys{WKL}_0$ доказуемо, что для всякой счетной модели $\mathcal M$ языка $\lng$ найдется счетная рекурсивно насыщенная модель $\mathcal N$, элементарно эквивалентная $\mathcal M$.
\end{cor}
\begin{proof}
Рассуждаем в $\sys{RCA}_0$. Существует совместное множество $$T_{\mathcal M } = \{\sigma \mid \mathcal St_{\lng(\T, C \upharpoonright M)}(\sigma) \wedge \mathcal M (\sigma) = 1\},$$
т.\,е. теория модели $\mathcal M$ в языке $\lng(C \upharpoonright M)$. Согласно лемме~\ref{rec_sat}, существует счетная рекурсивно насыщенная модель $\mathcal N' = (\mathcal N', N)$ множества $T_{\mathcal M '}$ в языке $\lng(C \upharpoonright M)$. Можно рассмотреть ее ограничение $\mathcal N ' = (\mathcal N, N)$ на язык $\lng$, положив $\mathcal N = \mathcal N \mid_{St_{\lng(C \upharpoonright N)} \cup Tm_{\lng(C \upharpoonright N)}}$. Модель $\mathcal N$ является искомой. Рекурсивная насыщенность очевидна. Теперь допустим, что $\sigma \in St_{\lng}$ и $\mathcal N \models \sigma$. Имеем $\mathcal N ' \models \phi$. Предположим, что $\mathcal M \not\models \phi$. Тогда $\mathcal M \models \neg\phi$ и, по определению, $\mathcal N' \models \neg\phi$. Противоречие. Отсюда видно, что модель $\mathcal M$ элементарно эквивалентна $\mathcal N$. 
\end{proof}

Пусть дана некоторая формула с одной свободной переменной $\theta(x)$ в смысле модели $\mathcal{M}$, т.\,е. некоторый терм языка $\lng(C \upharpoonright M)$, имеющий определенный код. Заметим, что тип $(\neg\Box^S_n[\theta(x)])_{n \in \omega}$ является рекурсивным (с параметром $\theta$ из $M$): действительно, $\gn{\neg\Box^S_n[\theta]} = sub(\gn{\neg\Box^S_n[y]}, \theta)$. Следовательно, существует рекурсивная нумерация $\{h\}$ этого типа: $\{h\}(n, \theta) = \gn{\neg\Box^S_n[\theta]}$.
\begin{lemma}\label{type}
Пусть модель $\mathcal M$ языка $\lng$ рекурсивно насыщена, $\mathcal M(St_{\lng{}}(\exists x\, \psi(x))) = 1$ и $\mathcal M(\Box^S_m[\theta \vee \neg\psi(c_{\{f\}(h,\, \theta \vee \neg\psi(x))})]) = 1$  для некоторого $m \in \omega$. Тогда найдется $m' \in \omega$, такое что
$$\mathcal M(\Box^S_{m'}[\theta \vee \neg\exists x\, \psi(x)]) = 1$$
\end{lemma}
\begin{proof}
Из условия и рекурсивной насыщенности $\mathcal M$ следует, что тип $(\neg\Box^S_n[\theta \vee \neg\psi(x)])_{n \in \omega}$ не является локально выполнимым. Тогда при некотором $m'' \in \omega$ для каждого $z \in M$ существует $n < m''$ со свойством
$$\mathcal M(\Box^S_n[\theta \vee \neg\psi(c_z)]) = 1.$$
С помощью леммы~\ref{gamma} получаем
$$\begin{array}{l}
\mathcal M(\Box^S_{m''}[\theta \vee \neg\psi(c_z)]) = 1 \text{ для всех $z \in M$};\\
\mathcal M(\Box^S_{m''+1}[\theta \vee \forall x\, \neg\psi(x)]) = 1;\\
\mathcal M(\Box^S_{m''+1}[\theta \vee \neg\exists x\, \psi(x)]) = 1.\\
\end{array}$$
\end{proof}

Пусть имеется некоторая модель $\mathcal M'$ теории $Q$. Согласно следствию~\ref{rec_sat_eq}, $\sys{WKL}_0$ доказывает существование счетной рекурсивно насыщенной модели $\mathcal M$, элементарно эквивалентной $\mathcal M'$.

Рассмотрим $\mathcal M$-рекурсивную нумерацию всех формул $\lng$ с одной свободной переменной в смысле модели $\mathcal M$: $(\psi_n(x_n))_{n \in \omega}$, так что для всех $n$ имеем
$$\mathcal M(St_{\lng}(\exists x_n\,\psi_n(x_n))) = 1.$$
Введем рекуррентную последовательность натуральных чисел $(a_i)_{i \in \omega}$, положив $a_0 = 1$ и $a_{i+1} = a_i(a_i + 1)$. Определим следующую $\mathcal M$-рекурсивную последовательность констант $c_{ij} \in C \upharpoonright M$, где для всякого $i$ выполнено $0\leqslant j \leqslant a_i$. Будем для простоты обозначать $C(u) = c_{\{f\}(h, u)}$ и положим
$$c_{00} = C(\neg\psi_0).$$
Ясно, что имеется ровно $a_{i + 1}$ наборов $(v_0,\ldots, v_{i})$, где каждое $v_k \in \{0, a_k\}$. При данном $i + 1$ можно отождествить каждый такой набор с его порядковым номером в лексикографическом упорядочении по возрастанию. Так, для всякого $j \in \{0, a_{i + 1}\}$ можно вычислить компоненты $j$-ого по порядку набора $(j_0,\ldots, j_{i})$. Предположим, что определены $c_{kj}$ для всех $k \leqslant i$ и соответствующих $j$. Для всех $k \leqslant i$ и $j \in [0, a_{k+1}]$ обозначим $(\psi_k)^j = \top$, если $j_k = 0$, и $ (\psi_k)^j = \psi_k(x / c_{k(j_k - 1)})$ в противном случае. Наконец, определим
$$c_{i+1\,j} = C(\bigvee_{k \leqslant i}\neg(\psi_k)^{j_k} \vee \neg\psi_{i+1}),$$
для всех $j \in [0, a_{k+1}]$.

Теперь мы готовы определить $\mathcal M$-рекурсивную последовательность формул $(F_n)_{n\in \omega}$:
$$\exists x_n \psi_n(x_n) \to \bigvee_{0 \leqslant j \leqslant a_n}\psi_n(x_n / c_{nj}).$$

\sys{RCA_0} доказывает существование $\mathcal M$-рекурсивного множества
$$A_{\mathcal M} = \{\phi \mid  \mathcal M(\Box^S_0[\phi]) = 1 \vee \exists n < \phi\, (\phi = F_n)\}.$$

Докажем, что $A_{\mathcal M}$ совместно в смысле \sys{PC}. Допустим противное. Тогда имеем
$$D, F_{n_1},\ldots, F_{n_w} \vdash \bot,$$
для некоторого конечного множества $D$, такого что для всякой $\phi \in D$ имеем $\mathcal M(\Box^S_0[\phi]) = 1$. Станем без ограничения общности считать, что $n_1 < \ldots < n_w$, и переименуем все переменные под квантором в формулах $F_{n_i}$ в какую-то одну. Будем также писать $\models \phi$, если  $M(\Box^S_m[\phi]) = 1$ для некоторого $m$.

Тогда согласно лемме~\ref{gamma} имеем $\models \wedge D$ и $\models\wedge D \to  \neg F_{n_1} \vee \ldots \vee \neg F_{n_w}$ и, окончательно, $\models \neg F_{n_1} \vee \ldots \vee \neg F_{n_w}$, т.\,е.
\begin{equation*}
\models (\exists x\,\psi_{n_1}(x) \wedge \bigwedge_{0 \leqslant j \leqslant a_{n_1}}\neg\psi_{n_1}(c_{n_1j}))\vee \ldots \vee(\exists x\,\psi_{n_w}(x) \wedge \bigwedge_{0 \leqslant j \leqslant a_{n_w}}\neg\psi_{n_w}(c_{n_w j})).
\end{equation*}
Переименуем для удобства наши формулы, положив $\phi_i = \psi_{n_i}$ для всех $i \in [1, w]$. Тогда
\begin{equation}\label{eq3}
\models (\exists x\,\phi_1(x) \wedge \bigwedge_{0 \leqslant j \leqslant a_{n_1}}\neg\phi_1(c_{n_1j}))\vee \ldots \vee(\exists x\,\phi_w(x) \wedge \bigwedge_{0 \leqslant j \leqslant a_{n_w}}\neg\phi_w(c_{n_w j})).
\end{equation}

Будем рассматривать последовательности чисел из $\{ 1,\ldots,w \}$. Для непустой подпоследовательности $(m_1,\ldots,m_v)$ назовем последовательность формул
$$\neg\phi_{m_1}(c_{m_1 j_1}),\ldots, \neg\phi_{m_v}(c_{m_v j_v})$$
\emph{правильной}, если константы согласованы в следующем смысле:
$$c_{m_{s+1} j_{s+1}} = C(\neg\phi_{m_1}(c_{m_1 j_1}) \vee \ldots \vee \neg\phi_{m_s}(c_{m_s j_s}) \vee \neg\phi_{m_{s+1}}).$$
Поскольку для каждой последовательности чисел $(m_1,\ldots,m_v)$ правильная последовательность однозначно определена, мы будем обозначим ее так: $$\neg\phi_{m_1}^*,\ldots. \neg\phi_{m_v}^*.$$
Если последовательность чисел пуста, то поставим ей в соответствие пустую правильную последовательность.

На множестве возрастающих последовательностей чисел из $\{1,\ldots,w\}$ рассмотрим лексикографический порядок $\prec$, полагая $(m_1,\ldots, m_v) \prec\linebreak (l_1, \ldots, l_u)$, если существует $r \in \omega$, такое что (1) $m_i = l_i$ для всех $i < r$, и (2) $m_r < l_r$ или $r = u < v$. Заметим, что если $(m_1,\ldots, m_v) \prec (l_1, \ldots, l_u)$, то для соответствующего $r$ будем иметь
$$(\neg\phi_{m_1}^*,\ldots, \neg\phi_{m_{r-1}}^*) = (\neg\phi_{l_1}^*,\ldots, \neg\phi_{l_{r-1}}^*).$$
Ясно также, что $\prec$ фундировано и, в частности, последовательность $(1,\ldots,w)$ будет наименьшей, а пустая "--- наибольшей.

\begin{stat}
Для всякой последовательности $(m_1,\ldots,m_v)$ выполнено $\models \neg\phi_{m_1}^* \vee \ldots \vee \neg\phi_{m_v}^*$.
\end{stat}
\begin{proof}
Рассуждаем по индукции. Для последовательности $(1,\ldots,w)$ получаем требуемое непосредственно из формулы~(\ref{eq3}). Допустим, что утверждение установлено для всех последовательностей $\prec$-меньших $(m_1,\ldots,m_v)$. Представим правильную последовательность $\neg\phi_{m_1}^*, \ldots.\neg\phi_{m_v}^*$ в виде
\begin{multline*}
\neg\phi_{l_1}^*, \neg\phi_{l_1 + 1}^*,\ldots, \neg\phi_{l_1 + b_1}^*, \neg\phi_{l_2}^*, \ldots, \neg\phi_{l_2 + b_2}^*, \ldots,\\ \neg\phi_{l_{u-1} + b_{u-1}}^*, \neg\phi_{l_u}^*, \ldots, \neg\phi_{l_u + b_u}^*.
\end{multline*}
Из~(\ref{eq3}) получаем
\begin{multline}\label{eq4}
\models \exists x\,\phi_{1}(x) \vee \ldots \vee\exists x\,\phi_{d_0}(x) \vee \neg\phi_{l_1}^* \vee \neg\phi_{l_1 + 1}^*\vee\ldots\vee \neg\phi_{l_1 + b_1}^* \vee\ldots\\
\exists x\,\phi_{l_1 + b_1+1}(x) \vee \ldots \vee\exists x\,\phi_{l_1 + b_1+d_1}(x) \vee \neg\phi_{l_2}^*\vee \ldots\vee \neg\phi_{l_2 + b_2}^*\vee \ldots\\
\neg\phi_{l_{u-1}}^* \vee\ldots\vee \neg\phi_{l_{u-1} + b_{u-1}}^* \vee \exists x\,\phi_{l_{u-1} + b_{u-1}+1}(x) \vee \ldots \vee \exists x\,\phi_{l_{u-1} + b_{u-1}+d_{u-1}}(x) \vee\ldots\\
\neg\phi_{l_u}^*\vee \ldots \vee \neg\phi_{l_u + b_u}^* \vee \exists x\,\phi_{l_{u} + b_{u}+1}(x) \vee \ldots \vee \exists x\,\phi_{l_{u} + b_{u}+d_{u}}(x),
\end{multline}
где $d_0 + 1 = l_1$, $l_i + b_i + d_i + 1 = l_{i+1}$ для всех $i \in [1, u)$, и $l_u + b_u + d_u = w$, причем может оказаться, что $d_0 = 0$ или $d_u = 0$, а также рассматривая правильная последовательность может быть пустой "--- в последнем случае в~(\ref{eq4}) входят только экзистенциальные формулы. Полученную дизъюнкцию можно естественным образом разбить на блоки, состоящие из одних лишь негативных или экзистенциальных формул. Эти блоки мы для краткости обозначим $N_i$ и $E_i$:
$$\models E_0 \vee N_{1} \vee E_{1} \vee \ldots \vee E_{{u-1}} \vee N_{u} \vee E_{u}.$$

Нашей целью будет устранить в формуле~(\ref{eq4}) все блоки экзистенциальных формул. Вновь рассуждаем индуктивно. Устраним блок $E_u$ (быть может, пустой). По предположению внешней индукции имеем
$$\models N_{1} \vee \ldots \vee N_{u-1} \vee \neg\phi_{l_u}^*\vee \ldots\vee \neg\phi_{l_u + b_u}^* \vee \neg\phi_{l_u + b_u + i}^*$$
для всех $i \in [1, d_u]$, ибо $l_u + b_u = m_v$ и $(m_1,\ldots m_v, m_v + i) \prec (m_1,\ldots m_v)$. Применяя лемму~\ref{type}, заключаем для каждого $i \in [1, d_u]$
$$\models \neg\phi_{l_1}^* \vee \ldots \vee \neg\phi_{l_1 + b_1}^* \vee \ldots\vee\neg\phi_{l_u}^*\vee \ldots\vee \neg\phi_{l_u + b_u}^* \vee \neg\exists x\,\phi_{l_u + b_u + i}(x),$$
что совместно с~(\ref{eq4}) дает требуемое
$$
\models E_0 \vee N_{1} \vee E_{1} \vee N_{2} \vee E_{2} \vee \ldots E_{{u-1}} \vee \neg\phi_{l_u}^*\vee \ldots \vee \neg\phi_{l_u + b_u}^*.
$$
Предположим теперь, что последние $\exists$-блоки, начиная с $(q+1)$-ого, устранены (или пусты), то есть выполнено
\begin{multline}\label{eq5}
\models E_0 \vee N_{1} \vee E_{1} \vee \ldots \vee N_{q-1} \vee E_{{q - 1}} \vee \neg\phi_{l_q}^* \vee \ldots \vee \neg\phi_{l_q + b_q}^* \vee \\
\exists x\,\phi_{l_{q} + b_{q}+1}(x) \vee \ldots \vee \exists x\,\phi_{l_{q} + b_{q}+d_{q}}(x) \vee\\
\neg\phi_{l_{q+1}}^* \vee \ldots \vee \neg\phi_{l_{q+1} + b_{q+1}}^* \vee N_{{q+2}} \vee \ldots \vee N_{u}.
\end{multline}
Если $l_q + b_q = m_r$ и $l_{q+1} = l_q + b_q + d_q + 1 = m_s$, то, очевидно, $(m_1, \ldots, m_r, m_r + i) \prec (m_1, \ldots, m_r, m_s, \ldots m_v)$ для всех $i \in [1, d_q]$. Значит, предположение внешней индукции дает
\begin{multline*}
\models N_{1} \vee \ldots \vee N_{{q - 1}} \vee \neg\phi_{l_q}^* \vee \ldots \vee \neg\phi_{l_q + b_q}^* \vee \neg\phi_{l_{q} + b_{q}+i}^*
\end{multline*}
для каждого $i \in [1, d_q]$. По лемме~\ref{type} получаем отсюда
\begin{multline*}
\models N_{1} \vee \ldots \vee N_{{q - 1}} \vee \neg\phi_{l_q}^* \vee \ldots \vee \neg\phi_{l_q + b_q}^* \vee \neg\exists\,\phi_{l_{q} + b_{q}+i}(x).
\end{multline*}
С помощью предположения внутренней индукции~(\ref{eq5}) заключаем требуемое
\begin{multline*}\label{eq5}
\models E_0 \vee N_{1} \vee E_{1} \vee \ldots \vee N_{q-1} \vee E_{{q - 1}} \vee \neg\phi_{l_q}^* \vee \ldots \vee \neg\phi_{l_q + b_q}^* \vee \\
N_{q+1} \vee N_{{q+2}} \vee \ldots \vee N_{u}.
\end{multline*}
Окончательно получаем $\models N_{1} \vee \ldots \vee N_{u}$, т.\,е. $\models \phi_{m_1} \vee \ldots \vee \phi_{m_v}$.
\end{proof}

Из доказанного утверждения вытекает, что $\models \bot$: достаточно рассмотреть пустую последовательность. Однако, это противоречит $\mathcal M \models Q$ и определению теории $Q$. Значит, множество $A_{\mathcal M}$ совместно.


Определим подмножество универсального двоичного дерева $B_{\mathcal M} \subseteq 2^{< \omega}$:
\begin{multline*}
t \in B_p \leftrightarrow \forall \sigma < l(t)\,  ((t)_\sigma = 1 \to \mathcal M (St_\lng(\sigma)) = 1) \wedge\\
\forall \sigma < l(t)\, (\sigma \in A_{\mathcal M} \to (t)_\sigma = 1 ) \wedge \\
\forall \sigma, \tau < l(t)\, ((\mathcal M(St_{\EuScript{L}}(\sigma)) = 1 \wedge \tau = \neg \sigma) \to (t)_\sigma = 1 - (t)_\tau) \wedge\\
\forall d < l(t)\, [(\mbox{"<$d$ есть вывод в \sys{CPC}">} \wedge\\ \forall i < l(d)\, (\mbox{"<$(d)_i$ есть собственная аксиома">} \to (t)_{(d)_i} = 1)) \to\\ \forall i < l(d)\, (\mathcal M(St_\lng((d)_i)) = 1 \to (t)_{(d)_i} = 1)]
\end{multline*}
С помощью принципа $\Sigma^0_0$-свертки в $\sys{RCA}_0$ доказуемо существование такого множества.
\begin{lemma}
В $\sys{RCA}_0$ доказуемо, что $B_{\mathcal M}$ есть бесконечное дерево.
\end{lemma}
\begin{proof}
Пусть $t \in B_{\mathcal M}$ и $s \sqsubset t$. Докажем, что $s \in B_{\mathcal M}$. Действительно, если $\sigma, \neg\sigma < l(s)$, то  имеем $\sigma, \neg\sigma < l(t)$ и $\min((s)_\sigma, (s)_{\neg\sigma}) = \min((t)_\sigma, (t)_{\neg\sigma}) = 0$. Проверим дедуктивную замкнутость. Пусть всем собственным аксиомам вывода $d < l(s)$ соответствуют единицы в последовательности $s$. По свойству геделевой нумерации $(d)_i < d$ для всех $i < l(d)$. Тогда из $l(s) < l(t)$ следует, что $(s)_{(d)_i} = (t)_{(d)_i} = 1$, если $\mathcal M(St_\lng((d)_i)) = 1$, для всех $i < l(d)$. Следовательно, $s \in B_{\mathcal M}$ и $B_{\mathcal M}$ является деревом.

Покажем, что для всякого $k < \omega$ найдется последовательность $t_k \in B_p$ такая, что $l(t_k) = k$. Построим $t_k$ посредством следующей рекурсивной процедуры. Полагаем вначале $t_k = (0, \ldots, 0)$. Затем для каждого $\phi < k$ кладем $(t_k)_\phi = 1$, если $\phi \in A_{\mathcal M}$. Замыкаем $t_k$ относительно выводов длины менее $k$ ($k$-замыкание), заменяя соответствующие нули на единицы. Ясно, что множество $D = \{\sigma | (t_k)_\sigma = 1\}$ совместно. Однако, возможно для некоторых $\phi, \neg\phi < k$ имеет место $\mathcal M(St_\lng(\phi)) = 1$ и $(t_k)_\phi = (t_k)_{\neg\phi} = 0$. Пусть $\phi_1, \ldots, \phi_m$ суть все такие формулы. Для каждого из $2^m$ возможных наборов строим $k$-замыкание $\wedge D \wedge \phi^{s_1}_1 \wedge \ldots \wedge \phi^{s_m}_m$, где $s_i \in {0, 1}$ и $\eta^0 = \neg\eta$, $\eta^1 = \eta$. Для какого-то набора $(s_1, \ldots, s_m)$ может оказаться, что для каких-либо $\psi, \neg\psi < k$ имеет место $\mathcal M(St_\lng(\psi)) = 1$ и $(t_k)_\psi = (t_k)_{\neg\psi} = 
1$. Однако, найдется набор $(s'_1, \ldots, s'_m)$, для которого это не так.
 
Действительно, допустим противное. Имеем $D, \phi^{s_1}_1, \ldots, \phi^{s_m}_m \vdash \bot$ для всех $(s_1, \ldots, s_m) \in \{0, 1\}^m$. Тогда, поскольку $\sys{PC} \vdash \bigvee_{(s_1, \ldots, s_m) \in \{0, 1\}^m} \phi^{s_1}_1 \wedge \ldots \wedge \phi^{s_m}_m \leftrightarrow (\phi_1 \vee \neg\phi_1) \wedge\ldots\wedge (\phi_m \vee \neg\phi_m)$, получаем $D \vdash \bot$, что не так.

Следовательно, положив $(t_k)_{\phi_i} = s'_i$ и $(t_k)_{\neg\phi_i} = 1 - s'_i$ и построив затем $k$-замыкание, мы получим последовательность $t_k \in B_{\mathcal M}$.
\end{proof}

С помощью аксиомы $(WKL)$ получаем, что в дереве $B_{\mathcal M}$ есть путь $p$. В $\sys{RCA}_0$ определим множество $\mathcal T = \{\sigma | (p)_\sigma = 1\}$ Очевидно, $\sys{RCA}_0$ доказывает, что $\mathcal T$ совместно, полно ($\forall \sigma [\mathcal M(St_{\EuScript{L}}(\sigma)) = 1 \to (t)_\sigma = 1 \vee (t)_{\neg\sigma} = 1]$) и дедуктивно замкнуто в смысле \sys{PC}, а также, что выполнено
\begin{equation}\label{eq6}
\forall \phi\, (\mathcal M(\Box^S_0[\phi]) = 1 \to \phi \in \mathcal T) \wedge \forall n\, (F_n \in \mathcal T).
\end{equation}

Рассмотрим модель $\mathcal N = (\mathcal M, \mathcal T)$ для языка $\lng(\T)$, полагая $\mathcal N(\T(\phi)) = 1 \rightleftharpoons \phi \in \mathcal T$.

\begin{lemma}
В $\sys{RCA}_0$ доказуемо, что в $\mathcal N$ истинны условия Тарского.
\end{lemma}
\begin{proof}\rule{1pt}{0mm}
\begin{enumerate}
\item Допустим $\mathcal M (At_{\EuScript{L}}(\phi)) = 1$. Если $\mathcal M (Tr_0[\phi]) = 1$, то в силу доказуемой $\Sigma_1$-полноты имеем $\mathcal M (\Box^S_0[\phi]) = 1$, откуда по~(\ref{eq6}) получаем $\phi \in \mathcal T$.

Обратно. Имеем $\phi \in \mathcal T$. Допустим $\mathcal M (Tr_0[\phi]) = 0$. Тогда $\mathcal M (\neg Tr_0[\phi]) = \mathcal M ( Tr_0[\neg\phi]) = 1$ вследствие полноты $\mathcal T$. По доказанному заключаем $\neg\phi \in \mathcal T$, что противоречит совместности $\mathcal T$.

\item Пусть $\phi, \psi \in \mathcal T$. Тогда $\phi \wedge \psi \in \mathcal T$ в силу дедуктивной замкнутости $\mathcal T$. Аналогично и в другую сторону.
\item Пусть $\phi[x / d] \in \mathcal T$ для всех $d \in C \upharpoonright M$. Предположим, что $\forall x\, \phi(x) \notin \mathcal T$. Тогда вследствие полноты $\mathcal T$, $\exists x\, \neg\phi(x) \in \mathcal T$. Теперь $\mathcal M$-рекурсивно вычислим $n$, такое что $\neg\phi = \psi_n$. Поскольку $F_n \in \mathcal T$, то имеем $\neg\phi[x / c] \in \mathcal T$ для некоторого $c \in C \upharpoonright M$  в силу дедуктивной замкнутости и полноты. Следовательно, $\mathcal T$ несовместно, что не так. Значит, $\forall x\, \phi(x) \in \mathcal T$. 

Обратно: пусть $\forall x\, \phi(x) \in \mathcal T$. Допустим, что существует $c \in \mathcal C\upharpoonright M$, такая что $\phi[x/c] \notin \mathcal T$. Тогда $\neg\phi[x/c] \in \mathcal T$ вследствие полноты. Отсюда по дедуктивной замкнутости получаем $\exists x\, \neg\phi(x) \in \mathcal T$, что противоречит совместности $\mathcal T$.
\end{enumerate}
\end{proof}

Остается проверить $\mathcal N \models R_{\omega+1}(S)$. Требуемое немедленно следует из~(\ref{eq6}).


\end{document}